\documentclass[11 pt]{amsart}
\usepackage[pdftex]{graphicx}
\usepackage{setspace, fullpage, amsthm, amsmath, amsfonts, amssymb, url}
\title{The Least-Perimeter Partition of a Sphere into Four Equal Areas}
\author{Max Engelstein}
\date{\today}
\address{Department of Mathematics, Yale University, New Haven, CT 06520}
\email{max.engelstein@yale.edu}
\begin{document}

\newtheorem{existence}{Theorem}[section]
\newtheorem{regularity}[existence]{Theorem}
\newtheorem{stability}[existence]{Theorem}
\newtheorem{definitions1}[existence]{Definition}
\newtheorem{definitions2}[existence]{Definition}
\newtheorem{assumegeodesic}[existence]{Theorem}
\newtheorem{assumer1conn}[existence]{Theorem}
\newtheorem{definitions3}[existence]{Definition}
\newtheorem{nodigons}[existence]{Corollary}
\newtheorem{oddcycle}[existence]{Lemma}
\newtheorem{circlebound}{Lemma}[section]
\newtheorem{partitionbound}[circlebound]{Corollary}
\newtheorem{numericalestimate}[circlebound]{Corollary}
\newtheorem{numericalestimatetwo}[circlebound]{Proposition}
\newtheorem{squarerootinequality}[circlebound]{Lemma}
\newtheorem{quinnsgaussbonnet}[circlebound]{Lemma}
\newtheorem{maxselimination}[circlebound]{Corollary}
\newtheorem{firstcurvaturebound}{Lemma}[section]
\newtheorem{rtwotriangles}[firstcurvaturebound]{Proposition}
\newtheorem{nottoolong}[firstcurvaturebound]{Corollary}
\newtheorem{nottoolong3}[firstcurvaturebound]{Corollary}
\newtheorem{generalizedtriangleinequality}[firstcurvaturebound]{Lemma}
\newtheorem{secondcurvaturebound}[firstcurvaturebound]{Lemma}
\newtheorem{rthreetriangles}[firstcurvaturebound]{Proposition}
\newtheorem{r4curvaturebound}{Lemma}[section]
\newtheorem{trianglesides}[r4curvaturebound]{Lemma}
\newtheorem{maintheorem}[r4curvaturebound]{Theorem}

\begin{abstract}
We prove that the least-perimeter partition of the sphere into four regions of equal area is a tetrahedral partition.
\end{abstract}

\maketitle
\tableofcontents

\section{Introduction}\label{sec: intro}
The spherical partition problem asks for the least-perimeter
partition of $\mathbb S^2$ into $n$ regions of equal area. The
corresponding planar ``Honeycomb Conjecture," open since antiquity
and finally proven by Hales \cite{honeycombplane} in 2001, states
that the regular hexagonal tiling provides a least-perimeter way to
partition the plane into unit areas. There are five analogous
partitions of the sphere into congruent, regular spherical polygons
meeting in threes (see Figure \ref{fig: geodesic nets}; for why the
edges must meet in threes, see Theorem \ref{thm: existence}), three
of which have already been proven minimizing: $n=2$, a great circle
(Bernstein \cite{bernstein}), $n=3$, three great semi-circles
meeting at 120 degrees at antipodal points (Masters \cite{masters}),
and $n=12$, a dodecahedral arrangement (Hales \cite{dodecahedral}).
The other two, the $n=4$ tetrahedral and the $n=6$ cubical
partitions, were conjectured to be minimizing. In this paper we
prove the $n=4$ conjecture:

\vspace{.1 cm}

\noindent {\bf Theorem \ref{thm: maintheorem}.} {\it The
least-perimeter partition of the sphere into four equal areas is the
tetrahedral partition.}

\vspace{.1 cm}

The main difficulty is that in principle each region may have many
components. Earlier results by Fejes T\'{o}th \cite{toth}, Quinn
\cite{quinn}, and Engelstein {\it et al}. \cite{emmp} required
additional assumptions to avoid a proliferation of cases.

Our approach starts with easy estimates to show that each region
must have one component that encloses the bulk of the area in that
region (Proposition \ref{thm: numerical estimates two}). Examination
of the curvature of the interfaces leads to the result that three of
the four regions must contain a triangle with large area (Corollary
\ref{thm: maxs elimination}, Propositions \ref{thm: r two triangles}
and \ref{thm: r three triangles}). Finally in Section \ref{sec:
final section} we examine the fourth region and conclude that the
tetrahedral partition is minimizing (Theorem \ref{thm:
maintheorem}).

Our proof requires little background knowledge beyond what is
discussed in Section \ref{sec: history and definitions}. With the
exception of the Gauss-Bonnet theorem, all the ideas and techniques
presented are covered by an introductory calculus course.

{\bf Acknowledgements:} This work began with the 2007 Williams
College National Science Foundation SMALL undergraduate research
Geometry Group. The author thanks the NSF and Williams College for
their support of the program. He thanks the 2007 Geometry Group
(Anthony Marcuccio, Quinn Maurmann, and Taryn Pritchard). He thanks
Conor Quinn and Edward Newkirk for their continuing involvement and
Nate Harman for his completion of the proof of Proposition \ref{thm:
numerical estimates two}. He thanks Sean Howe and Kestutis
Cesnavicius for their helpful comments on earlier drafts. He thanks
Yale University for travel support. Finally the author would like to
thank Frank Morgan, under whose guidance this research began and
without whose patience and persistence its completion would not have
been possible.

\section{Background and definitions}\label{sec: history and definitions}
Before delving into the particulars of the $n=4$ case we recall more
general results on the existence and regularity of minimizers.

\begin{existence}[Existence: \cite{soapbubbles}, Thm. 2.3 and Cor. 3.3]
\label{thm: existence} Given a smooth compact Riemannian surface $M$
and finitely many positive areas $A_i$ summing to the total area of
$M$, there is a least-perimeter partition of $M$ into regions of
area $A_i$. It is given by finitely many constant-curvature curves
meeting in threes at 120 degrees at finitely many points.
\end{existence}

It is important to note here that the edges of a minimizing
partition are not assumed to be geodesic. In fact, Lamarle
\cite{lamarle} and Heppes \cite{heppes} proved that there are only
ten nets of geodesics meeting in threes at 120 degrees on the
sphere. These nets are depicted in Figure \ref{fig: geodesic nets}
and include the previously proved minimizers for $n=2$, $n=3$, and
$n=12$ and the conjectured minimizers for $n=4$ and $n=6$. For other
values of $n$ the solution cannot be geodesic polygons. However,
Maurmann {\it et al.} \cite{memp} did show that, asymptotically, the
perimeter of the solution to the spherical partition problem
approaches that of the hexagonal tiling on the plane as $n$
approaches infinity.

\begin{figure}[b]
\begin{center}\scalebox{2}{\includegraphics{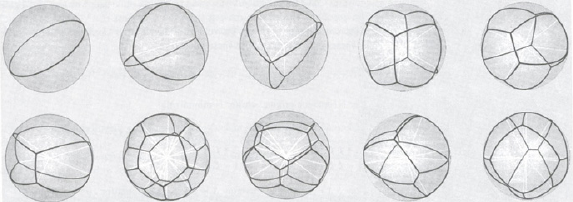}}\end{center}
\caption{The ten partitions of the sphere by geodesics meeting in
threes at $120$ degrees (picture originally from Almgren and Taylor
\cite{geodesicnets}, \copyright 1976 Scientific American).}
\label{fig: geodesic nets}
\end{figure}

A further regularity condition involves the concept of pressure:

\begin{stability}[\cite{quinn}, Prop. 2.5]
\label{thm: stability} In a perimeter-minimizing partition each
region has a pressure, defined up to addition of a constant, so that
the difference in pressure between regions $A$ and $B$ is the sum of
the (signed) curvatures crossed by any path from the interior of $B$
to the interior of $A$.
\end{stability}

\begin{definitions1}
\label{def: one} {\normalfont Following Quinn \cite{quinn}, we refer
to one highest-pressure region as $R_1$, and then in order of
decreasing pressure $R_2, R_3$ and $R_4$. Let $\kappa_{ij}$ be equal
to the pressure of $R_i$ minus the pressure of $R_j$. Note that
$\kappa_{ij} \geq 0$ if $i < j$ and $\kappa_{ij} = -\kappa_{ji}$.
Theorems \ref{thm: existence} and \ref{thm: stability} imply that
every edge between $R_i$ and $R_j$ has (signed) curvature
$\kappa_{ij}$.}
\end{definitions1}

With such strong combinatorial and geometric restrictions on
perimeter-minimizing partitions it may be tempting to dismiss the
spherical partition problem as a simple exercise in case analysis.
The crux of the difficulty (as we mentioned in the introduction) is
that disconnected regions are allowed. That regions can, {\it a
priori}, have a finite arbitrary number of components renders a
n\"{a}ive case analysis almost impossible. On the other hand, under
the strong assumption that each region is convex, Fejes T\'{o}th
\cite{toth} proved that each of the partitions in Figure \ref{fig:
geodesic nets} is minimizing for the areas that it encloses (this
also follows easily from the classification of geodesic nets of
Figure \ref{fig: geodesic nets}). For the case of $n=4$, Conor Quinn
proved the following, stronger result:

\begin{assumer1conn}[\cite{quinn} see Thm. 5.2]
\label{thm: assume r1 conn} In a perimeter-minimizing partition of
the sphere into four equal areas, if $R_1$ is connected, then that
partition is tetrahedral.
\end{assumer1conn}

This suggests suggest a more focused analysis on the components of
$R_1$. In order to avoid confusion we clarify some of our
terminology in this manner.

\begin{definitions3}
\label{def: three} {\normalfont In this paper an} $m$-{\it gon}
{\normalfont refers to a spherical polygon with $m$ edges, each with
constant curvature. We write} {\it digon} {\normalfont instead of
$2$-gon and often use the colloquial triangle, quadrilateral, or
pentagon for $3$-gon, $4$-gon, or $5$-gon. Finally we may abuse
terminology and use $m$-gon to refer to both the polygon and the
region bounded by that polygon (allowing us to refer to the ``area"
of an $m$-gon).}
\end{definitions3}

Before we delve into the analysis let us recall two more results.
The first is due to Quinn \cite{quinn} and is a corollary of Theorem
\ref{thm: existence}.

\begin{nodigons}[\cite{quinn} Lemma 2.11]
\label{thm: digons} A perimeter-minimizing partition of the sphere
does not contain a set of components whose union is a digon, with
distinct incident edges.
\end{nodigons}

From this it easily follows that in a non-tetrahedral partition no
two triangles share an edge. Our second result hinges on the easy
observation that no two components of the same region may share an
edge.

\begin{oddcycle}
\label{thm: oddcycle} In a perimeter-minimizing partition of the
sphere into four equal areas any component with an odd number of
sides is incident to at least one component from every other region.
\end{oddcycle}

Specifically a triangle is adjacent to exactly one component from
every other region. With all this in mind we can now move on to the
numerical analysis of Section \ref{sec: num bounds}.

\section{Area bounds}\label{sec: num bounds}

In this section we show that every region must consist of one large
component and then perhaps several small components (Proposition
\ref{thm: numerical estimates two}). Using the isoperimetric
inequality on the sphere and the length of the tetrahedral partition
we are able to establish strict upper bounds on the perimeter of any
one region in a potential minimizer. Our starting point is the
famous isoperimetric inequality of Bernstein.

\begin{circlebound}\cite{bernstein}
\label{thm: circlebound} For given area $0<A<4\pi$, a curve
enclosing area $A$ on the unit sphere has perimeter $P \geq B(A)=
\sqrt{A(4\pi-A)}$, with equality only for a single circle.
\end{circlebound}

Note that Lemma \ref{thm: circlebound} gives a lower bound for the
perimeter of a region with area $A$ even when the region is
comprised of several connected components.

\begin{partitionbound}
\label{thm: partitionbound} Given a partition of the sphere into $n$
equal areas, the total perimeter of the partition is greater than
$2\pi\sqrt{n-1}$.
\end{partitionbound}

\begin{proof}
Each region contains area $4\pi/n$. By Lemma \ref{thm: circlebound}
each region must have perimeter at least $B(4\pi/n)$. Multiply by
$n$ for the number of regions and divide by two (as each edge is
incident to at most two regions). Simplifying yields the desired
result.
\end{proof}

For $n = 4$, Corollary \ref{thm: partitionbound} yields that the
least-perimeter way to partition a sphere into four equal areas must
have perimeter at least $2\pi\sqrt{3} > 10.88$, whereas the
tetrahedral partition has perimeter $6\arccos(-1/3) < 11.47$ (given
by trigonometry). This yields an immediate upper bound on the size
of any one region.

\begin{numericalestimate}
\label{thm: numerical estimate}
In a perimeter-minimizing partition
of the sphere into four equal areas every region has perimeter less
than $6.62$.
\end{numericalestimate}

\begin{proof}
Let $x$ be the perimeter of some region. By Lemma \ref{thm:
circlebound} the perimeter $P$ of the entire partition satisfies $P
> (1/2)(x + 3\pi\sqrt{3})$. Yet if the partition is minimizing then
we have $P < 11.47$. Numerics yield $x < 6.62$, the desired result.
\end{proof}

We will now prove and apply an inequality which will force any
region to have one large component (Proposition \ref{thm: numerical
estimates two}).

\begin{squarerootinequality}
\label{thm: square root inequality} The function (for fixed $0 < k
\leq 2\pi$) $$f_k(t) = \sqrt{t(4\pi-t)} + \sqrt{(k-t)(4\pi-k+t)}$$
defined on the interval $[0,k]$ is symmetric about the point $t =
k/2$, and $f_k'(t) > 0$ for all $0 < t < k/2$.
\end{squarerootinequality}

\begin{proof}
It is evident that the function is symmetric about $t =k/2$. The
radicands are downward parabolas (in $t$), so the sum of their
square roots is a concave down function. Symmetry implies the
desired result.
\end{proof}

\begin{numericalestimatetwo}
\label{thm: numerical estimates two} In a perimeter-minimizing
partition of the sphere into four equal areas, every region must
contain a component with area at least $23\pi/25$.
\end{numericalestimatetwo}

\begin{proof}
Let $t$ be the area of the largest component in the given region. By
Lemma \ref{thm: circlebound} we have the inequality $P(t) \geq B(t)
+ B(\pi-t),$ where $P$ is the perimeter of the region. Corollary
\ref{thm: numerical estimate} yields $6.62 > B(t) + B(\pi-t)$. On
the other hand, setting $t = 23\pi/25$ gives $$P(\frac{23\pi}{25})
\geq B(\frac{23\pi}{25}) + B(\frac{2\pi}{25}) =
\frac{\pi}{25}(\sqrt{23\cdot77} + 14) \approx 7 > 6.62.$$ So Lemma
\ref{thm: square root inequality} says $t > 23\pi/25$ or $t <
2\pi/25$.

Suppose $t < 2\pi/25$. By Lemma \ref{thm: circlebound} when $A =
2\pi/25$, the region must have perimeter greater than $14\pi/25 =
7A$. As $B(x)$ is concave down we have $B(x) \geq 7x$ for $x <
2\pi/25$. Therefore the perimeter of the region is greater than 7
times the area of the region. So $t < 2\pi/25$ implies that the
perimeter of the region is at least $7\pi \approx 21.99 > 6.62$, a
clear contradiction of Corollary \ref{thm: numerical estimate}.
\end{proof}

The following Lemma \ref{thm: conor quinn theorem} due to Quinn
\cite{quinn} will produce a large triangle in $R_1$ (Corollary
\ref{thm: maxs elimination}).

\begin{quinnsgaussbonnet}[\cite{quinn}, Lemma 5.12]
\label{thm: conor quinn theorem} In the highest-pressure region of a
perimeter-minimizing partition, (1) a triangle must have area less
than or equal to $\pi$, (2) a square must have area less than or
equal to $2\pi/3$, (3) a pentagon must have area less than or equal
to $\pi/3$, and (4) all other polygons cannot exist. Equality can
only occur when the polygon is geodesic.
\end{quinnsgaussbonnet}

\begin{proof}
The result follows directly from Gauss-Bonnet and the convexity of
the components of $R_1$.
\end{proof}

\begin{maxselimination}
\label{thm: maxs elimination} In a perimeter-minimizing partition of
the sphere into four equal areas, $R_1$ must contain a triangle, and
this triangle must have area at least $23\pi/25$.
\end{maxselimination}

\begin{proof}
The result is immediate from Lemma \ref{thm: conor quinn theorem}
and Proposition \ref{thm: numerical estimates two}.
\end{proof}

\section{$R_2$ and $R_3$ contain large triangles}\label{sec: r2 and r3 triangles}

The goals of this section are Propositions \ref{thm: r two
triangles} and \ref{thm: r three triangles}: that $R_2$ and $R_3$
each contains a triangle with area at least $23\pi/25$. The
possibility that the components of $R_2$ or $R_3$ are not convex
prohibits us from using Lemma \ref{thm: conor quinn theorem} and
necessitates a closer look at the curvature of the interfaces. We
start off by bounding $\kappa_{12}$.

\begin{firstcurvaturebound}
\label{thm: r two curvature bound}
In a least-perimeter partition of the sphere into four equal areas, $\kappa_{12} < 1/21$.
\end{firstcurvaturebound}

\begin{proof}
By Corollary \ref{thm: maxs elimination}, $R_1$ has a triangle, $T$,
of area $A_T \geq 23\pi/25$. By Gauss-Bonnet, the perimeter $P$ and
exterior angles $\alpha_i$, of this triangle satisfy $$2\pi = A_T +
\int_{\partial T}\kappa ds + \sum \alpha_i \geq \frac{23\pi}{25} +
P\kappa_{12} + \pi,$$  so $P\kappa_{12}\leq 2\pi/25.$ By Lemma
\ref{thm: circlebound}, $P \geq B(23\pi/25)$. Therefore $\kappa_{12}
< 1/21$.
\end{proof}

Now we are able to establish an analogue to Corollary \ref{thm: maxs elimination} for $R_2$.

\begin{rtwotriangles}
\label{thm: r two triangles} In a least-perimeter partition of the
sphere into four equal areas, $R_2$ contains a triangle of area at
least $23\pi/25$.
\end{rtwotriangles}

\begin{proof}
Proposition \ref{thm: numerical estimates two} says that $R_2$ must
have a component with area no less than $23\pi/25$. Assume by way of
contradiction that this component has at least four sides. Then
Gauss-Bonnet gives $$\frac{23\pi}{25} - \kappa_{12}P_{12} +
\kappa_{r}P_r + \frac{4\pi}{3} \leq 2\pi,$$ where $P_{12}$ is the
perimeter between $R_1$ and $R_2$. $\kappa_rP_r$ represents the
(positive) contribution of curvature from lower pressure regions.
Combining terms we get: $\pi(23/25 -2/3) \leq \kappa_{12}P_{12}$.
Bounding $\kappa_{12}$ using Lemma \ref{thm: r two curvature bound}
yields $19\pi/75 \leq P_{12}/21.$ Isolating $P_{12}$ gives $P_{12}
\geq 133\pi/25 > 12$, obviously contradicting Corollary \ref{thm:
numerical estimate}.
\end{proof}

In order to establish an analogue to Lemma \ref{thm: r two curvature
bound} for $R_3$, we must insure that $R_2$ does not occupy too much
of the perimeter of $R_1$'s large triangle. A quick corollary bounds
the length of the side of $R_1$'s large triangle which is incident
to $R_2$ (a side we know exists by Lemma \ref{thm: oddcycle}).

\begin{nottoolong}
\label{thm: not too long} In a least-perimeter partition of the
sphere into four equal areas, let $P$ be the perimeter of the large
triangle in $R_1$, and let $l$ be the length of the side incident to
$R_2$ in that triangle. Then $l \leq P/3$.
\end{nottoolong}

\begin{proof}
If the partition in question is tetrahedral, then the statement is
trivial. Assume that it is not tetrahedral and, to obtain a
contradiction, that $l > P/3$. By Corollary \ref{thm: digons} no two
triangles are incident to one another in a non-tetrahedral
minimizing partition. Therefore the perimeter of $R_2$ is at least
$P_2 + l$, where $P_2$ is the perimeter of the large triangle in
$R_2$ (whose existence is established in Proposition \ref{thm: r two
triangles}). By Lemma \ref{thm: circlebound} we have that the
perimeter of $R_2$ is at least $(4/3)B(23\pi/25) > 7$, which
contradicts the partition's minimality by Corollary \ref{thm:
numerical estimate}.
\end{proof}

Now we proceed as in the $R_2$ case; first we bound $\kappa_{13}$.

\begin{secondcurvaturebound}
\label{thm: second curvature bound}
In a least-perimeter partition of the sphere into four equal areas, $\kappa_{13} < 1/14$.
\end{secondcurvaturebound}

\begin{proof}
By Corollary \ref{thm: maxs elimination} $R_1$ has a large triangle.
Let $P$ be the perimeter of this triangle and $P'$ the lengths of
the side of the triangle which are not incident to $R_2$. Then
Corollary \ref{thm: not too long} states that  $P \leq 3P'/2$.
Using Lemma \ref{thm: circlebound} to obtain a lower bound for $P$
we write $P' \geq (2/3) B(23\pi/25).$ Applying Gauss-Bonnet to this
large triangle gives the inequality $23\pi/25 + \kappa_{13}P' \leq
\pi.$ Substituting the bound on $P'$ and simplifying results in the
desired inequality $\kappa_{13} \leq 3/\sqrt{23\cdot77} < 1/14$.
\end{proof}

In the same vein as Proposition \ref{thm: r two triangles} we now
prove that $R_3$ must contain a triangle with area at least
$23\pi/25$.

\begin{rthreetriangles}
\label{thm: r three triangles} In a least-perimeter partition of the
sphere into four equal areas, $R_3$ contains a triangle with area at
least $23\pi/25$.
\end{rthreetriangles}

\begin{proof}
By Proposition \ref{thm: numerical estimates two} $R_3$ must contain
some component with area at least $23\pi/25$. For the sake of
contradiction, assume that component has at least four sides. Let
$P$ be the perimeter of this component, then Gauss-Bonnet yields
$$\frac{23\pi}{25} -\kappa_{13}P + \frac{4\pi}{3} \leq 2\pi,$$ or
$\kappa_{13}P \geq 19\pi/75.$ By Lemma \ref{thm: second curvature
bound} we get that $\kappa_{13} < 1/14$, which means that $P >
266\pi/75 > 11$  which contradicts the partition's minimality by
Corollary \ref{thm: numerical estimate}.
\end{proof}

Now we can establish an analogue to Corollary \ref{thm: not too long} for $R_3$.

\begin{nottoolong3}
\label{thm: not too long 3} In a least-perimeter partition of the
sphere into four equal areas, let $P$ be the perimeter of the large
triangle in $R_1$, and let $l$ be the length of the side incident to
$R_3$ in that triangle. Then $l \leq P/3$.
\end{nottoolong3}

\begin{proof}
If the partition is tetrahedral, then the statement is trivial.
Assume that it is not tetrahedral and, to obtain a contradiction,
that $l > P/3$. By Corollary \ref{thm: digons} no two triangles are
incident to one another in a non-tetrahedral minimizing partition.
Therefore the perimeter of $R_3$ is at least $P_3 + l$, where $P_3$
is the perimeter of the large triangle in $R_3$ (whose existence is
established in Proposition \ref{thm: r three triangles}). By Lemma
\ref{thm: circlebound} we have that the perimeter of $R_3$ is at
least $(4/3)B(23\pi/25) > 7$, which contradicts the partition's
minimality by Corollary \ref{thm: numerical estimate}.
\end{proof}

\section{The tetrahedral partition is minimizing}\label{sec: final section}
In this section we reach our goal in Theorem \ref{thm: maintheorem},
which states that the perimeter-minimizing partition of the sphere
into four equal areas is the tetrahedral partition. We require only
one lemma, a lower bound for $\kappa_{14}$.

\begin{r4curvaturebound}
\label{thm: r four curvature bound} In a non-tetrahedral
perimeter-minimizing partition of the sphere into four equal areas
we have $\kappa_{14} > 1/2$.
\end{r4curvaturebound}

\begin{proof}
Let $P$ be the perimeter of the large triangle in $R_1$ (which we
know exists by Corollary \ref{thm: maxs elimination}) and $P_r$ be
the rest of the perimeter of $R_1$. Corollary \ref{thm: numerical
estimate} then gives $$6.62 > P + P_r \geq B(23\pi/25) + P_r$$ where
the second inequality is Lemma \ref{thm: circlebound}. This yields
$P_r < 1.34$.

Since the partition is non-tetrahedral, by Theorem \ref{thm: assume
r1 conn} $R_1$ must have another component, and this component has
at most five sides (Lemma \ref{thm: conor quinn theorem}). Use
Gauss-Bonnet on this component to obtain a second inequality on
$P_r$: $$\kappa_{14}P_r\geq \frac{\pi}{3} - \frac{2\pi}{25}$$ and
combine the two inequalities to get that $\kappa_{14} > 1/2$.
\end{proof}

Now we reach our ultimate goal.

\begin{maintheorem}
\label{thm: maintheorem}
The least-perimeter partition of the sphere
into four equal areas is the tetrahedral partition.
\end{maintheorem}

\begin{proof}
Assume that $R_1$ is non-tetrahedral. Let $P$ be the perimeter of
the large triangle in $R_1$ (which exists by Corollary \ref{thm:
maxs elimination}), and let $l_4$ be the length of the side of this
large triangle incident to $R_4$. By Corollaries \ref{thm: not too
long} and \ref{thm: not too long 3} and Lemma \ref{thm: circlebound}
we have $l_4 \geq P/3 \geq (1/3)B(23\pi/25)$.

Apply Gauss-Bonnet to the large triangle in $R_1$ to get
$$\frac{23\pi}{25} + \frac{\kappa_{14}B(\frac{23\pi}{25})}{3} \leq
\pi.$$ Simplify and isolate $\kappa_{14}$ to obtain $\kappa_{14}
\leq 6/\sqrt{23\cdot 77} < 1/7$, a clear contradiction of Lemma
\ref{thm: r four curvature bound}.
\end{proof}

\end{document}